\documentclass[10pt]{article}
\usepackage{amsmath,amssymb,graphicx,amsthm, mathrsfs}
\usepackage{tikz}
\usetikzlibrary{matrix, arrows}

\newcommand{\charg}[3]{{{\rm H}}^{#1}({#2},{#3})}
\newtheorem{theorem}{Theorem}[section]
\newtheorem{lemma}[theorem]{Lemma}
\newtheorem{corollary}[theorem]{Corollary}
\newtheorem{proposition}[theorem]{Proposition}
\newtheorem{remark}[theorem]{Remark}

\newtheorem{example}[theorem]{Example}
\begin{document}
%-----------------------------------Insert----------------------------------%

\title{Universal deformation rings for extensions of finite subgoups of $Gl_2(\mathbb{C})$}

\author{David C. Meyer\\Department of Mathematics\\University of Missouri\\
Columbia, MO 65211\\ \textsf{meyerdc@missouri.edu}}

\date{}

\maketitle

\begin{abstract}
In this paper we expand on previous results, studying the extent to which one can detect fusion in certain finite groups $\Gamma$, from information about the universal deformation rings of absolutely irreducible $\mathbb{F}_p\Gamma$-modules.  We consider groups which are extensions of finite irreducible subgroups of $Gl_2({\mathbb{C}})$ by elementary abelian $p$-groups of rank $2$.
\end{abstract}

\section{Introduction}
\label{s:intro}

In this paper we consider the function which assigns to an absolutely irreducible $\mathbb{F}_p\Gamma$-module $V$, its universal deformation ring $R(\Gamma,V)$.  We determine precisely when this function is nonconstant on the collection of two-dimensional $\mathbb{F}_p\Gamma$-modules.  Then, we show that when the function is nonconstant, we can use its graph to determine information about the internal structure of the group $\Gamma$.  Specifically, we connect the fusion of certain subgroups $N$ of $\Gamma$, to the kernels of those representations whose corresponding modules are a level set of the function $V \to R(\Gamma,V)$.  

We say that a pair of elements of $N$ are fused in $\Gamma$, if they are conjugate in $\Gamma$ but not in $N$, and we call the collection of elements of $N$ which are fused in $\Gamma$, the fusion of $N$ in $\Gamma$.  The universal deformation ring $R(\Gamma,V)$, is characterized by the property that the isomorphism class of every lift of $V$ over a complete local commutative Noetherian ring $R$ with residue field $\mathbb{F}_p$ arises from a unique local ring homomorphism $ R(\Gamma,V)\xrightarrow{\alpha} R$.  When the function $V \to R(\Gamma,V)$ is nonconstant, our goal  is to determine how to transfer information about universal deformation rings to information about the internal structure of the group.  One might expect a connection with fusion because it plays a key role in the character theory of $\Gamma$, which in turn enters into finding universal deformation rings of representations.

We will consider groups $\Gamma$ which are extensions of finite subgroups of $Gl_2({\mathbb{C}})$ by elementary abelian $p$-groups of rank $2$.  For any group $\Gamma$, let $\Sigma(\Gamma)$ denote the collection of isomorphism classes of absolutely irreducible two-dimensional $\mathbb{F}_p\Gamma$-modules.

We can now state our main results:

\begin{theorem}
\label{Th1}
Let $G$ be a finite irreducible subgroup of  $Gl_2({\mathbb{C}})$.  Let $p$ be an odd prime such that 
${\mathbb{F}}_p G$ is semisimple, and ${\mathbb{F}}_p$ is a sufficiently large field for $G$.  Let $\phi$ be an irreducible action of $G$ on $N = {\mathbb{Z}}/p{\mathbb{Z}}\times {\mathbb{Z}}/p{\mathbb{Z}}$.  Let $\Gamma = {\Gamma}_{\phi}$ be the corresponding semidirect product.  Then, the following two statements are equivalent,
\begin{itemize}
\item[i.]
$\phi$ is trivial on the center of $G$
\item[ii.]
there exists a $V$ in $\Sigma(\Gamma)$ with $R(\Gamma,V) \ncong {\mathbb{Z}}_p$.
\end{itemize}
\end{theorem}

\begin{theorem}
\label{Th2}
Let $G$ be a finite irreducible subgroup of  $Gl_2({\mathbb{C}})$.  Let $p$ be an odd prime such that ${\mathbb{F}}_p G$ is semisimple, and ${\mathbb{F}}_p$ is a sufficiently large field for $G$.   Let $\phi$ be an irreducible action of $G$ on $N = {\mathbb{Z}}/p{\mathbb{Z}}\times {\mathbb{Z}}/p{\mathbb{Z}}$, and let $\Gamma = {\Gamma}_{\phi}$ be the corresponding semidirect product.  Suppose that $\phi$ is trivial on the center of $G$.  Then one can determine the fusion of $N$ in $\Gamma$ from the set $\{ker(\rho) : \rho \in \Sigma(\Gamma) \textrm{ with } R(\Gamma,V_{\rho}) \ncong \mathbb{Z}_p\}$.
\end{theorem}

Theorem \ref{Th1} says that the function 
$$\Sigma({\Gamma}_{\phi}) \to \{ \textrm{local rings} \}, \textrm{ which sends }V \to R({\Gamma}_{\phi},V)$$
is nonconstant in this context exactly when the representation $\phi$ is trivial on the center of $G$.  The necessity of this condition was established in \cite[Cor. 3.4]{meyer}.  Theorem \ref{Th2} says that when the function $V \to R({\Gamma}_{\phi},V)$ is not trivial, knowledge of its graph can be used to determine the  fusion of $N$ in $\Gamma$.  Both Theorems \ref{Th1} and \ref{Th2} generalize the results in the analysis of dihedral groups in \cite{meyer}.

This paper is organized as follows.  In section \ref{s:prelim} we recall the basic definitions of deformations and deformation rings.  In section \ref{groups} we make preliminary statements about groups which admit faithful two-dimensional irreducible complex representations.  In sections \ref{dihedral} and section \ref{udr} we focus on those finite subgroups of $Gl_2({\mathbb{C}})$ which are central extensions of dihedral groups by cyclic groups.  Then, in section \ref{proof} we prove our main results.  

This paper is a continuation of work begun in my dissertation at the University of Iowa under the supervision of Professor Frauke Bleher.  I would like to thank her for all of her advice and guidance.
\section{Preliminaries on universal deformation rings}
\label{s:prelim}

In this section, we give a brief introduction to deformations and universal deformation rings.  We focus specifically on the perfect prime field ${\mathbb{F}}_p$.  For more background material, we refer the reader to \cite{mazur} and \cite{desmit-lenstra}.

Let  $p$ be an odd prime, $\mathbb{F}_p$ be the field with $p$ elements, and let $\mathbb{Z}_p$ denote the ring of $p$-adic integers.  Let $\hat{\mathcal{C}}$ be the category whose objects are tuples $(R, {\pi}_R)$, where $R$ is a complete local commutative Noetherian rings with residue field 
$\mathbb{F}_p$, and ${\pi}_R$ is a projection to the residue field. The morphisms in $\hat{\mathcal{C}}$ are local homomorphisms of local rings that induce the identity map on $\mathbb{F}_p$.

Suppose $\Gamma$ is a finite group and $V$ is a finitely generated $\mathbb{F}_p\Gamma$-module. 
A lift of $V$ over an object $R$ in $\hat{\mathcal{C}}$ is a pair $(M,\phi)$ where $M$ is a finitely generated 
$R\Gamma$-module that is free over $R$, and $\phi$ is an isomorphism of $\mathbb{F}_p\Gamma$-modules $$\mathbb{F}_p\otimes_R M\xrightarrow{\phi} V.$$
Two lifts $(M,\phi)$ and $(M',\phi')$ of $V$ over $R$ are isomorphic if there is an $R\Gamma$-module isomorphism $M\xrightarrow{\alpha} M'$ with $\phi=\phi'\circ (\mathrm{id}_{\mathbb{F}_p}\otimes\alpha)$. The 
isomorphism class $[M,\phi]$ of a lift $(M,\phi)$ of $V$ over $R$ is called a deformation of $V$ over $R$, and 
the set of such deformations is denoted by $\mathrm{Def}_\Gamma(V,R)$. The deformation functor is the functor
$\hat{\mathcal{C}} \xrightarrow{\hat{F}_V} \mathrm{Sets}$, which sends an object 
$$R \in \hat{\mathcal{C}} \xrightarrow{\hat{F}_V} \mathrm{Def}_\Gamma(V,R) \in \textrm{Sets}$$
and a morphism 
$$[R\xrightarrow{f} R']\in {\textrm{Hom}}_{\hat{\mathcal{C}}}(R,R') \xrightarrow{\hat{F}_V} \textrm{ the function } [[M,\phi]\xrightarrow{{\hat{F}_V}f} [R'\otimes_{R,f} M,\phi']].$$ 
In the above, the map $\mathrm{Def}_\Gamma(V,R) \xrightarrow{{\hat{F}_V}f} \mathrm{Def}_\Gamma(V,R')$ is given by 
$$[M,\phi]\to [R'\otimes_{R,f} M,\phi'],$$
where  $\phi'=\phi$ after identifying 
$\mathbb{F}_p\otimes_{R'}(R'\otimes_{R,f} M)$ with $\mathbb{F}_p\otimes_R M$.\\

If there exists an object $R(\Gamma,V)$ in $\hat{\mathcal{C}}$ and a deformation $[U(\Gamma,V),\phi_U]$ of 
$V$ over $R(\Gamma,V)$ such that for each $R$ in $\hat{\mathcal{C}}$, and for each lift $(M,\phi)$ of $V$ over 
$R$ there is a unique morphism $R(\Gamma,V)\xrightarrow{\alpha} R$ in $\hat{\mathcal{C}}$ such that 
$$\hat{F}_V(\alpha)([U(\Gamma,V),\phi_U])=[M,\phi],$$
then we call $R(\Gamma,V)$ the universal deformation 
ring of $V$ and $[U(\Gamma,V),\phi_U]$  the universal deformation of $V$. 

In other words, $R(\Gamma,V)$ represents the functor $\hat{F}_V$ in the sense that $\hat{F}_V$ is naturally isomorphic to 
$\mathrm{Hom}_{\hat{\mathcal{C}}}(R(\Gamma,V),-)$. In the case when the morphism 
$R(\Gamma,V)\xrightarrow{\alpha} R$ relative to the lift $(M,\phi)$ of $V$ over $R$ is only known to be unique if 
$R$ is the ring of dual numbers over $\mathbb{F}_p$ but may be not unique for other $R$, 
$R(\Gamma,V)$ is called the versal deformation ring of $V$ and 
$[U(\Gamma,V),\phi_U]$ is called the versal deformation of $V$.

By \cite{mazur}, every finitely generated ${\mathbb{F}}_p\Gamma$-module $V$ has a versal deformation ring $R(\Gamma,V)$. 
Moreover, if $V$ is an absolutely irreducible $\mathbb{F}_p\Gamma$-module, then $R(\Gamma,V)$ is
universal.

The following result of Mazur shows the connection between the structure of $R(\Gamma,V)$ and certain first and second cohomology
groups $\charg{i}{\Gamma}{\mathrm{Hom}_{\mathbb{F}_p}(V,V))}$, for $i = 1, 2.$

\begin{theorem} {\rm (\cite[\S1.6]{mazur}, \cite[Thm. 2.4]{bockle})}
\label{thm:udr}
Suppose $V$ is an absolutely irreducible $\mathbb{F}_p\Gamma$-module, and let
$d^i_V=\mathrm{dim}_{\mathbb{F}_p}\mathrm{H}^i(\Gamma,\mathrm{Hom}_{\mathbb{F}_p}(V,V))$ 
for $i=1,2$. Then $R(\Gamma,V)$ is isomorphic to 
a quotient algebra $\mathbb{Z}_p[[t_1,\ldots,t_r]]/J$ where $r=d^1_V$ and $d^2_V$ is an upper bound 
on the minimal number of generators of $J$.
\end{theorem}

\section{Preliminaries on groups with faithful irreducible two-dimensional complex representations}
\label{groups}
In this section we discuss groups $G$ that admit a faithful two-dimensional irreducible complex representation.  These are exactly the isomorphism classes of irreducible finite subgroups of $Gl_2({\mathbb{C}})$.  Examples of such groups include $Sl_2({\mathbb{F}}_3), Sl_2({\mathbb{F}}_5)$, the binary octahedral group, and dihedral groups.  For the remainder of this paper, let $D_{2n}$ denote the dihedral group of order $2n \geq 6$. 

If $G$ is a finite irreducible subgroup of $Gl_2({\mathbb{C}})$, we associate to the  isomorphism class of $G$ an odd prime $p$ such that 
\begin{itemize}
\item[i.]
${\mathbb{F}}_p G$ is semisimple, and
\item[ii.]
${\mathbb{F}}_p$ is a sufficiently large field for $G$.
\end{itemize}

Clearly, $p = p(G)$.  In this situation, the irreducible representation $G \subseteq Gl_2({\mathbb{C}})$ is realizable over ${\mathbb{F}}_p$ and $G$ is a semisimple subgroup of $Gl_2({\mathbb{F}}_p)$.  We now use the classical result that semisimple subgroups of $Gl_2(k)$ for $k$ finite, are parametrized by their image in $PGl_2(k)$ \cite{lange}.  More precisely, if $\pi$ is the natural projection from $Gl_2({\mathbb{F}}_p) \xrightarrow{\pi} PGl_2({\mathbb{F}}_p)$ then,

$$\pi(G) =
\begin{cases}
D_{2n}, \textrm{ for some } n \\
{\mathbb{Z}}/m{\mathbb{Z}}, \textrm{ for some } m\\
A_4,
A_5, \textrm{ or }
S_4.
\end{cases} $$
Now, for any irreducible two-dimensional ${\mathbb{F}}_p$ representation of $G$ (not necessarily faithful), we set $\Gamma$ equal to the corresponding semidirect product $$\Gamma ={\Gamma}_{\phi} =  ({\mathbb{Z}}/p{\mathbb{Z}}\times {\mathbb{Z}}/p{\mathbb{Z}}) {\rtimes}_{\phi} G.$$  Equivalently, we may say $\Gamma$ is any group with a  normal elementary abelian rank two Sylow $p$-subgroup such that the corresponding quotient group satisfies the technical conditions with respect to ${\mathbb{F}}_p$ stated above.  We will sometimes say the fusion of $\phi$, for the fusion of $N = {\mathbb{Z}}/p{\mathbb{Z}}\times {\mathbb{Z}}/p{\mathbb{Z}}$ in ${\Gamma}_{\phi}$.  For any such $\Gamma$, we obtain the sequence,

$$\begin{tikzpicture}[description/.style={fill=white,inner sep=2pt}]
\matrix (m) [matrix of math nodes, row sep=3em,
column sep=2.5em, text height=1.5ex, text depth=0.25ex]
{  0 & {\mathbb{Z}}/p{\mathbb{Z}}\times {\mathbb{Z}}/p{\mathbb{Z}} & \Gamma & G & 1\\ 
{} & {} & {} & H & {}\\ };
%\draw[double,double distance=5pt] (m-1-1) – (m-1-3);
\path[->,font=\scriptsize]

(m-1-1) edge node[auto] {} (m-1-2)
(m-1-2) edge node[auto] {$\iota$} (m-1-3)
(m-1-3) edge node[auto] {$\pi$} (m-1-4)
(m-1-4) edge node[auto] {} (m-1-5)
(m-1-4) edge [bend left = 60] node [above] {$\phi$} (m-1-2)
(m-1-4) edge node[auto] {$\pi$} (m-2-4);
\end{tikzpicture}$$

\noindent
where $\pi$ is the natural projection into $PGl_2({\mathbb{F}}_p)$, and $H$ is cyclic, dihedral, or one of $A_4, A_5, $or $S_4$.  Since the center of $Gl_2({\mathbb{F}}_p)$ is cyclic, we have that any such $G$ is a central extension of $H$ by a cyclic group.  Thus, we have the diagram

$$\begin{tikzpicture}[description/.style={fill=white,inner sep=2pt}]
\matrix (m) [matrix of math nodes, row sep=3em,
column sep=2.5em, text height=1.5ex, text depth=0.25ex]
{ {} & {} & {} & 0 & {}\\ {} & {} & {} & {\mathbb{Z}}/m{\mathbb{Z}} & {}\\ 0 & {\mathbb{Z}}/p{\mathbb{Z}}\times {\mathbb{Z}}/p{\mathbb{Z}} & \Gamma & G & 1\\ 
{} & {} & {} & H & {}\\ {} & {} & {} & 0 & {}\\ };
%\draw[double,double distance=5pt] (m-1-1) – (m-1-3);
\path[->,font=\scriptsize]

(m-2-4) edge node[auto] {$\iota$} (m-3-4)
(m-3-1) edge node[auto] {} (m-3-2)
(m-3-2) edge node[auto] {$\iota$} (m-3-3)
(m-3-3) edge node[auto] {$\pi$} (m-3-4)
(m-3-4) edge node[auto] {} (m-3-5)
(m-3-4) edge [bend left = 60] node [above] {$\phi$} (m-3-2)
(m-3-4) edge node[auto] {$\pi$} (m-4-4)
(m-1-4) edge node[auto] {} (m-2-4)
(m-4-4) edge node[auto] {} (m-5-4);
\end{tikzpicture}$$

where $m$ divides $ p-1$.  We point out that a central extension of $H$ by a cyclic group need not be isomorphic to a finite irreducible subgroup of $Gl_2({\mathbb{C}})$.
Fix $\Gamma = {\Gamma}_{\phi}$, and $p$ prime.  Let $\Sigma = \Sigma(\Gamma)$ denote the collection of isomorphism classes of absolutely irreducible two-dimensional ${\mathbb{F}}_p \Gamma$- modules.  Note that every absolutely irreducible ${\mathbb{F}}_p \Gamma$- module is an irreducible ${\mathbb{F}}_p G$-module inflated to $\Gamma$.  We will consider the function 
$$\Sigma \to \hat{\mathcal{C}}, \textrm{ which sends } V \to R(\Gamma,V).$$
Since $\phi$ is irreducible, if there exists a $V \in \Sigma(\Gamma)$ with  $R(\Gamma,V) \ncong {\mathbb{Z}}_p$, then $\phi$ must be trivial on the center of $G$ \cite[Corollary 3.4]{meyer}.  By \cite[Appendix by Feit]{lange}, $\pi(G)$ is dihedral or cyclic exactly when $G$ is contained in the normalizer of some Cartan subgroup $C$ of $Gl_2({\mathbb{F}}_p)$.  The cyclic case corresponds to $G \subseteq C$, and the dihedral case to $G \subseteq N(C), G \not\subseteq C$.  By our hypotheses on $G, p,$ and $\phi$, $G$ cannot be abelian, so $G$ may not be inside a Cartan subgroup (though an analysis on the connections between fusion and universal deformation rings for finite abelian groups was conducted in \cite[sec 4.6]{meyer}).  Because ${\mathbb{F}}_p$ is sufficiently large for $G$, we need only consider split Cartan subgroups if $\pi(G)$ is dihedral.  By Theorem \ref{thm:udr} \cite{mazur} to show $R(\Gamma,V) \ncong {\mathbb{Z}}_p$ it is sufficient to demonstrate that $\charg{1}{\Gamma}{\mathrm{Hom}_{\mathbb{F}_p}(V,V))}$ is not trivial.  We will use the  results of \cite{meyer} on group cohomology and universal deformation rings.

We will show first that the function $V \to R(\Gamma,V)$ is not the constant function $V \to {\mathbb{Z}}_p$  if and only if the representation $\phi$ is trivial on the center of $G$.  Moreover, we will show that when $\phi$ is trivial on $Z(G)$, knowledge of where this function takes values different from ${\mathbb{Z}}_p$ can be used to detect the fusion of ${\mathbb{Z}}/p{\mathbb{Z}}\times {\mathbb{Z}}/p{\mathbb{Z}} $ in $\Gamma$.  
That is, when the function $V \to R(\Gamma,V)$ is nonconstant, we obtain the correspondence

$$\textrm{ Fusion of } \phi \leftrightsquigarrow \{ker(\rho) : \rho \in \Sigma \textrm{ and }R(\Gamma,V_{\rho}) \ncong \mathbb{Z}_p \}.$$

The correspondence above means that given the fusion of $\phi$ one can determine the kernels of those representations whose corresponding ${\mathbb{F}}_p {\Gamma}_{\phi}$-modules have universal deformation ring different from the $p$-adic integers.  Moreover, this assignment is reversible.  Given the collection of kernels of representations whose corresponding modules have universal deformation ring not isomorphic to the $p$-adic integers, one knows the fusion of $\phi$. 

These results generalize the results in \cite{meyer}, though we suppress the explicit mention of cohomologically maximal modules.  The finite irreducible subgroup of $Gl_2({\mathbb{C}})$ which are central extensions of dihedral groups by cyclic groups play a particularly important role.  We point out that in \cite{meyer} the analysis was aided by the fact that the representations of dihedral groups are, of course, well known.  In contrast, while central extensions of dihedral groups by cyclic groups which admit faithful irreducible two-dimensional complex representations, can all be viewed as extensions, even their isomorphism classes are not readily accessible.  In the argument that follows we perform our analysis on fusion and universal deformation rings without the a priori knowledge of the representation theory of the groups. From a careful reading of \cite{meyer}, one can see the following correspondences are also obtained

\begin{align*}
\textrm{ Fusion of } \phi &\leftrightsquigarrow \{ker(\rho) : \rho \in \Sigma \textrm{ and }R(\Gamma,V_{\rho}) \cong {\mathbb{Z}}_{p}[[t]]/(t^2,pt)\}\\
\textrm{ Fusion of } \phi &\leftrightsquigarrow \{ker(\rho) : \rho \in \Sigma \textrm{ and } {\textrm{dim}}_{{\mathbb{F}}_p}(\charg{2}{\Gamma}{\textrm{Hom}_{\mathbb{F}_{p}}(V,V)}) = 2\}\\
\textrm{ Fusion of } \phi &\leftrightsquigarrow \{ker(\rho) : \rho \in \Sigma \textrm{ and } {\textrm{dim}}_{{\mathbb{F}}_p}(\charg{1}{\Gamma}{\textrm{Hom}_{\mathbb{F}_{p}}(V,V)}) = 1\}.
\end{align*}
\section{Central extensions of Dihedral groups by cyclic groups}
\label{dihedral}

In this section, we consider the case where $G$ is a finite irreducible subgroup of $Gl_2({\mathbb{C}})$ which is a central extension of a dihedral group by a cyclic group.  This corresponds to the case in section \ref{groups}, where $\pi(G) = H = D_{2n}$.  Examples of such groups include dihedral groups, semi-dihedral groups, and generalized quaternion groups.  Clearly, any such group is either dihedral, or a central extension of a dihedral group by a cyclic group.  We point out that this necessary condition is not sufficient, as many central extensions of dihedral groups by cyclic groups do not admit a faithful two-dimensional complex representation.  

Let $G$ be a central extension of $D_{2n}$ by ${\mathbb{Z}}/m{\mathbb{Z}}$.  Clearly, $G$ has a presentation 
$$G= \langle X, Y, Z | X^n = Z^{\alpha}, Y^2 = Z^{\beta}, YX = X^{n-1} YZ^{\gamma}, YZ = ZY,  XZ = ZX \rangle $$

where $\alpha, \beta, \gamma $ are viewed modulo $m$.  For $G$ with the presentation above we will write $G$ is given by $\langle X,Y,Z\rangle , (\alpha, \beta, \gamma)$.

First, without any hypothesis on the representation theory of $G$, we can suppose $\beta = 0,$ or $\beta =1.$

\begin{lemma}
Let $G$ be a central extension of a dihedral group by a cyclic group, then $G$ is given by $\langle X, Y, Z \rangle , (\alpha, \beta, \gamma)$, where $\beta = 0$ or $\beta = 1$.
\end{lemma}
\begin{proof}
Let $G$ be a central extension of $D_{2n}$ by the cyclic group ${\mathbb{Z}}/m{\mathbb{Z}}$.  Then, $G$ is given by $\langle X,Y,Z \rangle, (\alpha, \beta, \gamma)$, for some $(\alpha, \beta, \gamma) $.  First suppose $Z^{\beta} \in \langle Z^2 \rangle$.  Let $G'$ be the central extension of $D_{2n}$ by ${\mathbb{Z}}/m{\mathbb{Z}}$  given by $\langle X', Y', Z' \rangle$ corresponding to the triple $(\alpha, 0, \gamma)$.  Let $\Phi$ be given by 
$$X' \xrightarrow{\Phi} X, Y' \xrightarrow{\Phi} YZ^{-\tau}, Z' \xrightarrow{\Phi} Z,$$
where $\tau$ satisfies $\beta = 2 \tau (\textrm{ mod } m) $.  Since the images of $X', Y', Z' $ under $\Phi$ satisfy the relevant relations, $\Phi$ defines a homomorphism from $H$ to $G$.  By inspection, $\Psi$ given by 
$$X \xrightarrow{\Psi} X', Y \xrightarrow{\Psi} (Y')({(Z')}^{\tau}, Z \xrightarrow{\Psi} Z'$$
is its inverse.  Similarly, when $Z^{\beta} \notin \langle Z^2 \rangle$ one shows $G \cong H = \langle X', Y', Z' \rangle$ corresponding to the triple $(\alpha, 1, \gamma)$.  Of course the latter case is only possible when $m$ is even.
\end{proof}
For those central extensions of dihedral groups by cyclic groups with faithful irreducible two-dimensional complex representations, we next show that considering two-dimensional irreducible representation is natural, as such groups have only one and two-dimensional irreducible complex representations. 
\begin{lemma}
Let $G$ is a central extension of a dihedral group by a cyclic group which admits a faithful irreducible two-dimensional complex representation.  Say $G$ given by $\langle X, Y, Z \rangle , (\alpha, \beta, \gamma)$ as above.  Then, any irredicble representation of $G$ is one or two-dimensional.  
\end{lemma}
\begin{proof}
Let $\iota$ denote the faithful irreducible ${\mathbb{F}}_p$ representation, and let $\Gamma={\Gamma}_{\iota}$ be the corresponding semidirect product.  Clearly, $ N = {\mathbb{Z}}/p{\mathbb{Z}} \times {\mathbb{Z}}/p{\mathbb{Z}} $ is a normal Sylow $p$-subgroup of $\Gamma$.  Then,  $\Gamma/N = G$ has an abelian subgroup of index $2$ given by $\langle X,Z \rangle$.  Therefore, by \cite[Thm. 1]{lee} every irreducible representation of $\Gamma$ in the algebraic closure of ${\mathbb{F}}_p$,  $\overline{{\mathbb{F}}}_p$ has order $1$ or $2$.  Since ${\mathbb{F}}_p$ is sufficiently large, the same holds for ${\mathbb{F}}_p$.  Since an absolutely irreducible ${\mathbb{F}}_p \Gamma$- module is an irreducible ${\mathbb{F}}_p G$-module inflated to $\Gamma$, the result follows.
\end{proof}

We next show that if $G$ is a central extension of a dihedral group by a cyclic group which admits a faithful irreducible two-dimensional complex representation, then every irreducible two-dimensional ${\mathbb{F}}_p$ representation of $G$ can be taken into the normalizer of the diagonal in $Gl_2({\mathbb{F}}_p)$.  That is, for such a $G$, $\rho$ irreducible implies one can assume
$$\rho(G) \subseteq \left\{\begin{pmatrix} a & 0\\0 & b\end{pmatrix} | a, b \neq 0 \right\} \bigcup \left\{\begin{pmatrix} 0 & a\\ b & 0\end{pmatrix} | a, b \neq 0 \right\} = N(D). $$

\begin{lemma}
\label{N(D)}
Let $G$ be a central extension of a dihedral group by a cyclic group which admits a faithful irreducible two-dimensional complex representation, say $G$ is given by $\langle X, Y, Z \rangle , (\alpha, \beta, \gamma)$.  Let $p$ and ${\mathbb{F}}_p$ be as above.  Then, any irredicble two-dimensional ${\mathbb{F}}_p$ representation of $G$ is isomorphic as a representation to one with image in the subgroup of $Gl_2({\mathbb{F}}_p)$ consisting  of diagonal and skew diagonal matrices.   
\end{lemma}
\begin{proof}
Since  $G \subseteq Gl_2({\mathbb{F}}_p)$, with $\pi (G) \cong D_{2n}$, we begin with the short exact sequence below
$$\begin{tikzpicture}[description/.style={fill=white,inner sep=2pt}]
\matrix (m) [matrix of math nodes, row sep=3em,
column sep=2.5em, text height=1.5ex, text depth=0.25ex]
{ {} & {} & 0 & {} & {}\\ 0 & {\mathbb{Z}}/m{\mathbb{Z}} & G & D_{2n} & 1\\ 
{} & {} & Gl_{2}({\mathbb{F}}_p) & {} & {}\\ };
%\draw[double,double distance=5pt] (m-1-1) – (m-1-3);
\path[->,font=\scriptsize]

(m-1-3) edge node[auto] {} (m-2-3)
(m-2-1) edge node[auto] {} (m-2-2)
(m-2-2) edge node[auto] {$\iota$} (m-2-3)
(m-2-3) edge node[auto] {$\pi$} (m-2-4)
(m-2-4) edge node[auto] {} (m-2-5)
(m-2-3) edge node[auto] {$\iota$} (m-3-3);
\end{tikzpicture}$$
By the classification of semisimple subgroups of $Gl_2({\mathbb{F}}_p)$ \cite[Appendix by Feit]{lange}, $G$ is contained in the normalizer of a Cartan subgroup, but not the Cartan subgroup.  Because ${\mathbb{F}}_p$ is sufficiently large, it must be a split Cartan subgroup (one conjugate to ${\mathbb{F}}_p^* \times {\mathbb{F}}_p^*$) .  Therefore, by conjugating we may assume $G \subseteq N(D)$.  Now, for any irreducible two-dimensional representation $\rho$, we obtain the morphism of sequences below.
$$\begin{tikzpicture}[description/.style={fill=white,inner sep=2pt}]
\matrix (m) [matrix of math nodes, row sep=3em,
column sep=2.5em, text height=1.5ex, text depth=0.25ex]
{ \\ 0 & {\mathbb{Z}}/m{\mathbb{Z}} & G & D_{2n} & 1\\ 
0 & {\mathbb{F}}_p^* & Gl_{2}({\mathbb{F}}_p) & PGl_{2}({\mathbb{F}}_p)  & 1 \\ };
%\draw[double,double distance=5pt] (m-1-1) – (m-1-3);
\path[->,font=\scriptsize]

(m-2-1) edge node[auto] {} (m-2-2)
(m-2-2) edge node[auto] {$\iota$} (m-2-3)
(m-2-3) edge node[auto] {$\pi$} (m-2-4)
(m-2-4) edge node[auto] {} (m-2-5)
(m-2-3) edge node[auto] {$\rho$} (m-3-3)
(m-3-1) edge node[auto] {} (m-3-2)
(m-3-2) edge node[auto] {$\iota$} (m-3-3)
(m-3-3) edge node[auto] {$\pi$} (m-3-4)
(m-3-4) edge node[auto] {} (m-3-5)
(m-2-2) edge node[auto] {$\iota$} (m-3-2)
(m-2-4) edge node[auto] {$r$} (m-3-4);
\end{tikzpicture}$$

Let $\bar{r}$ be the injection $\displaystyle D_{2n}/\textrm{ ker }(r) \xrightarrow{\bar{r}} PGl_{2}({\mathbb{F}}_p)$.  We obtain the commuting square
$$\begin{tikzpicture}[description/.style={fill=white,inner sep=2pt}]
\matrix (m) [matrix of math nodes, row sep=3em,
column sep=2.5em, text height=1.5ex, text depth=0.25ex]
{ G/ker(\rho) & D_{2n}/ker(r)\\ 
Gl_{2}({\mathbb{F}}_p) &  PGl_{2}({\mathbb{F}}_p). \\ };
%\draw[double,double distance=5pt] (m-1-1) – (m-1-3);
\path[->,font=\scriptsize]

(m-1-1) edge node[auto] {} (m-1-2)
(m-1-1) edge node[auto] {$\bar{\rho}$} (m-2-1)
(m-2-1) edge node[auto] {$\pi$} (m-2-2)
(m-1-2) edge node[auto] {$\bar{r}$} (m-2-2);
\end{tikzpicture}$$

Thus, by applying the analysis in \cite[Appendix by Feit]{lange} again to the image, if the image of $r$ is cyclic, then $\rho(G) \subseteq D \subseteq N(D)$ and $\rho$ is not irreducible, so the image of $r$ is dihedral $\rho(G) \subseteq N(D)$.
\end{proof}

An easy computation shows that $\charg{2}{D_{2n}}{{\mathbb{Z}}/m{\mathbb{Z}}}$ depends on the parity on $n$ and $m$.  Specifically,
$$\charg{2}{D_{2n}}{{\mathbb{Z}}/m{\mathbb{Z}}} =
\begin{cases}
0, \textrm{ if } 2 \not| m\\
{\mathbb{Z}}/2{\mathbb{Z}}, \textrm{ if } 2 \not| n, 2|m\\
{\mathbb{Z}}/2{\mathbb{Z}}\times {\mathbb{Z}}/2{\mathbb{Z}} \times {\mathbb{Z}}/2{\mathbb{Z}}, \textrm{ if } 2|n, m.

\end{cases}$$

That is, independent of whether such a group shows up as an irreducible subgroups of $Gl_2({\mathbb{C}})$, there are at most $2$ isomorphism classes of central extensions of $D_{2n}$ by ${{\mathbb{Z}}/m{\mathbb{Z}}}$ in all cases except when both $n$ and $m$ are even.  If $n$ and $m$ are both even, there will always be extensions of $D_{2n}$ by ${{\mathbb{Z}}/m{\mathbb{Z}}}$ which admit faithful two-dimensional irreducible representations, for example central products of $D_{2n}$ with ${{\mathbb{Z}}/m{\mathbb{Z}}}$.  For the time being, we restrict our attention to this case.   Until section \ref{proof}, from this point forward we assume $2 |  n, m$.  We will now classify these groups in terms of the parameters corresponding to their presentation.

\begin{lemma}
\label{identity}
Let $G$ be given by $\langle X, Y, Z \rangle , (\alpha, \beta, \gamma)$ as before.  Let $n, m$ be even, $p$ be such that ${\mathbb{F}}_p$ is sufficiently large for $G$,  ${\mathbb{F}}_p G$ is semisimple.  Then, $G$ is an extension of $D_{2n}$ by $C_{m}$ which admits a faithful two-dimensional irreducible complex representation if and only if $\frac{n}{2}(\alpha + \gamma) = \alpha + \frac{m}{2} (\textrm{mod } m).$
\end{lemma}
\begin{proof}
Since $G$ is isomorphic to image under its faithful irreducible representation into $Gl_{2}({\mathbb{F}}_p)$, By Lemma \ref{N(D)}, $G$ is isomoprhic to the subgroup of $N(D)$ generated by
$$X = \begin{pmatrix} w^{a+i} & 0\\  0 & w^{a} \end{pmatrix}, Y =  \begin{pmatrix} 0 & w^r \\  w^r & 0 \end{pmatrix}, Z = \begin{pmatrix} w^t & 0\\  0 & w^t \end{pmatrix},$$
where $w$ a generator for ${\mathbb{F}}_p^*$, ${\mathbb{Z}}/m{\mathbb{Z}} \cong \langle w^t \rangle$, ${\mathbb{Z}}/n{\mathbb{Z}} \cong \langle w^i \rangle$ and $w^{2r} = 1$ if $\beta = 0$ and $w^t$ if $\beta = 1$.  By a simple computation, the relations
$$ X^n = Z^{\alpha}, Y^2 = Z^{\beta}, YX = X^{n-1} YZ^{\gamma}$$
are equivalent to showing the existence of exponents such that $w^{an} = w^{\alpha t}$ and $w^{2a+i}=w^{(\alpha + \gamma)t}$.  Now let $A=w^{\alpha t}, C=w^{\gamma t} $.  Clearly $A$ and $C$ are $m$-th roots of unity, and by inspection $\frac{n}{2}(\alpha + \gamma) = \alpha + \frac{m}{2} (\textrm{mod } m)$ if and only if $-A = {(AC)}^{\frac{n}{2}}$.  If $w^{na} = w^{\alpha t}$ and $w^{2a+i}=w^{(\alpha + \gamma)t}$, then $$w^{\frac{n}{2} (\alpha + \gamma )t}= w^{\frac{n}{2} 2a+i} = w^{\alpha t} w^{i\frac{n}{2}}= -w^{\alpha t}.$$  So  $-A = {(AC)}^{\frac{n}{2}}$.  Conversely, we show that we can solve $t^n = A, \omega t^2 = AC $.  First, since ${\mathbb{F}}_p$ is sufficiently large, we may solve for $t^2 = {\omega}^{-1} AC.$  Then since $\frac{n}{2}(\alpha + \gamma) = \alpha + \frac{m}{2} (\textrm{mod } m), t^{2\frac{n}{2}} = {(AC)}^{\frac{n}{2}}{{\omega}^{-1}}^{\frac{n}{2}}$, so $t^n = -{(AC)}^{\frac{n}{2}} = A $.  Thus, $G$ given by $\langle X, Y, Z \rangle , (\alpha, \beta, \gamma)$ if and only if  $\frac{n}{2}(\alpha + \gamma) = \alpha + \frac{m}{2} (\textrm{mod } m).$
\end{proof}
With these preliminary results on central extensions of dihedral groups by cyclic groups out of the way, we now concentrate on fusion and universal deformation rings.

\section{Universal deformation rings and fusion}
\label{udr}

In this section we generalize the results of \cite{meyer} on dihedral groups to central extensions of dihedral groups by cyclic groups which admit a faithful irreducible two-dimensional complex representation.  This section will go a long way towards the proofs of Theorems \ref{Th1} and \ref{Th2}.

For the remainder of this section, $G \subseteq Gl_2({\mathbb{C}})$, $G$ is finite, irreducible with $G$ a central extension of $D_{2n}$ and by ${\mathbb{Z}}/m{\mathbb{Z}}$, for $n, m$ even.  As in the previous section, when we use a presentation of $G$, we will say $G$ is given by $\langle X, Y, Z \rangle , (\alpha, \beta, \gamma)$.  We choose an odd prime $p = p(G)$ with ${\mathbb{F}}_p G$ is semisimple, and ${\mathbb{F}}_p$ is sufficiently large for $G$.  We will let $w$ denote a generator for the cyclic group ${\mathbb{F}}_p^*$.  Recall, that $\phi$ is any irreducible two-dimensional ${\mathbb{F}}_p$ representation of $G$ (not necessarily faithful), and $\Gamma = {\Gamma}_{\phi}$ is the corresponding semidirect product $({\mathbb{Z}}/p{\mathbb{Z}}\times {\mathbb{Z}}/p{\mathbb{Z}}) {\rtimes}_{\phi} G$.  

In this section, we will show that there exists an irreducible two-dimensional representation with universal deformation ring different from the $p$-adic integers if and only if $\phi$ is trivial on the center of $G$.  We will then show that when this condition is satisfied, representations $\phi$ with the same fusion will correspond to identical sets $\{ker(\rho), \textrm{where } R({\Gamma}_{{\phi}},V_{\rho}) \ncong \mathbb{Z}_p \}$.

\begin{lemma}
\label{every}
Let $G$ be as above, $G$ given by $\langle X, Y, Z \rangle , (\alpha, \beta, \gamma), \phi$ any irreducible two-dimensional ${\mathbb{F}}_p$ representation.  Then, there exists an absolutely irreducible two-dimensional ${\mathbb{F}}_p \Gamma$-module with universal deformation ring different from the $p$-adic integers if and only if $\phi$ is trivial on the center of $G$.  Moreover, the set of two-dimensional irreducible modules with universal deformation ring different from the $p$-adic integers is a full orbit under the action of the character group of $G$.
\end{lemma}
\begin{proof}
 Recall that a complete set of  isomorphism classes of 2-dimensional irreducible representations of $D_{2n}$ over $\mathbb{F}_p$ are given by

$$r\xrightarrow{\theta_{\ell}} \begin{pmatrix} \omega^{i\ell}&0\\ 0&\omega^{-i\ell} \end{pmatrix}, \hspace*{.2 in}
s\xrightarrow{\theta_{\ell}} \begin{pmatrix} 0&1\\   1&0 \end{pmatrix}$$
where $D_{2n}$ is given by $\langle r,s | r^n, s^2, srs^{-1}r \rangle,$ and $1 \leq \ell \leq \frac{n}{2} -1$. \\

First, say $\beta = 0$.  For any $\phi$ trivial on $Z(G)$ we will construct explicitly a representation with a non-trivial universal deformation ring.  By Lemma \ref{N(D)}, we may assume $$X = \begin{pmatrix} w^{a+i} & 0\\  0 & w^{a} \end{pmatrix}, Y =  \begin{pmatrix} 0 & w^r \\  w^r & 0 \end{pmatrix}, Z = \begin{pmatrix} w^t & 0\\  0 & w^t \end{pmatrix}, $$ 

where, 
\begin{eqnarray}
w^{an} = w^{t\alpha}, w^{2r}=1, w^{2a+i} = w^{t(\alpha + \gamma)}.
\end{eqnarray}

\smallskip

Since $\phi$ is trivial on $\langle Z \rangle$, $\phi = {\theta}_{\ell}$ for some $\ell$.  Explicitly, $\phi = {\theta}_{\ell}$ is given by

$$X \xrightarrow{{\theta}_{\ell} } \begin{pmatrix} w^{i \ell} & 0\\  0 & w^{-i \ell } \end{pmatrix}, Y \xrightarrow{{\theta}_{\ell} }  \begin{pmatrix} 0 & 1 \\  1 & 0 \end{pmatrix}, Z \xrightarrow{{\theta}_{\ell} } \begin{pmatrix} 1 & 0\\  0 & 1 \end{pmatrix}. $$

\smallskip
\noindent
We proceed as in \cite{meyer}, showing that for $\phi = {\theta}_{\ell}$, a representation $\theta$ of $G$ has universal deformation ring different from ${\mathbb{Z}}_p$ if and only if $V_{\phi}$ is a summand of $V_{\theta}^* \otimes V_{\theta} = {\textrm{Hom}}_{{\mathbb{F}}_p}(V_{\theta},V_{\theta})$.  By Lemma \ref{N(D)} we choose a basis such that with respect to this basis the image of $\theta$ is contained in $N(D)$.  By inspection, the one dimensional modules generated by each of $$\begin{pmatrix} 1 & 0\\0 & 1\end{pmatrix},\begin{pmatrix} 1 & 0\\0 & -1\end{pmatrix}$$

respectively, are each summands of ${\textrm{Hom}}_{{\mathbb{F}}_p}(V_{\theta},V_{\theta})$ with the adjoint action.  A simple computation shows that the span of 
$$\begin{pmatrix} 0 & 0\\1 & 0\end{pmatrix}, \begin{pmatrix} 0 & 1\\0 & 0\end{pmatrix}$$

is an indecomposable two-dimensional submodule of ${\textrm{Hom}}_{{\mathbb{F}}_p}(V_{\theta},V_{\theta})$, and is isomorphic to ${\theta}_{\ell}$ if and only if $X \xrightarrow{\theta}  \begin{pmatrix} w^e & 0\\0 & w^{e+ i \ell}\end{pmatrix}$.

\medskip 

With this observation in mind, we define a representation ${\rho}_{\ell}$ of $G$ by,
$$X \xrightarrow{{\rho}_{\ell}} \begin{pmatrix} w^{a\ell} & 0\\  0 & w^{a\ell + i \ell} \end{pmatrix}, Y \xrightarrow{{\rho}_{\ell}} \begin{pmatrix} 0 & w^r \\  w^r & 0 \end{pmatrix}, Z \xrightarrow{{\rho}_{\ell}}  \begin{pmatrix} w^{t\ell} & 0\\  0 & w^{t\ell} \end{pmatrix}$$

\smallskip

One can verify directly that using the identities in (1), that

\begin{eqnarray*}
({{\rho}_{\ell}(X)})^n &=& ({{\rho}_{\ell}(Z)})^{\alpha}, ({{\rho}_{\ell}(Y)})^2 = 1= ({{\rho}_{\ell}(Z)})^m, \\ {{\rho}_{\ell}(Y)} {{\rho}_{\ell}(X)}&=&({{\rho}_{\ell}(X)})^{n-1}{{\rho}_{\ell}(Y)}({{\rho}_{\ell}(Z)})^{\gamma}.
\end{eqnarray*}

Therefore, ${\rho}_{\ell}$ is a representation of $G$.  By inflation, ${\rho}_{\ell}$ is an absolutely irreducible ${\mathbb{F}}_p \Gamma$-module.   By observing the difference between the exponents of the diagonal entries of ${\rho}_{\ell}(X)$, we see that $R(\Gamma,V_{{\rho}_{\ell}})$ is ${\mathbb{Z}}_{p}[[t]]/(t^2,pt)$ \cite{meyer, bleher-chinburg-desmit}.  Moreover, for any representation $\tau$ with $R(\Gamma, V_{\tau}) \ncong {\mathbb{Z}}_p$, $\tau$ must send $$ X \xrightarrow{\tau}  \begin{pmatrix} w^f & 0\\  0 & w^{f+\ell i} \end{pmatrix}, Y \xrightarrow{\tau} \begin{pmatrix} 0 & w^s \\  w^s & 0 \end{pmatrix}, Z \xrightarrow{\tau} \begin{pmatrix} w^{tg} & 0\\  0 & w^{tg} \end{pmatrix} .$$  Since, by hypothesis, $\tau$ is a representation of $G$, it must be the case that
\begin{eqnarray}
 w^{2s} = 1, w^{nf} = w^{tg \alpha} = w^{nag}, w^{tg(\alpha + \gamma)} = w^{2f + \ell i} .  
\end{eqnarray}
Thus, we define $\chi$ by $$X \xrightarrow{\chi} w^{f- \ell a}, Y \xrightarrow{\chi} w^{s-r}, Z \xrightarrow{\chi} w^{tg-t \ell}.$$
Then, by direct computation using (1), (2), $$({\chi(X)})^n = ({\chi(Z)})^{\alpha}, ({\chi(Y)})^2 = 1= ({\chi(Z)})^m, {\chi(Y)}{\chi(X)}=({\chi(X)})^{n-1}{\chi(Y)}({\chi(Z)})^{\gamma}.$$  Thus, $\chi$ is a character of $G$.  Since the action of the character group preserves the difference between the exponents of the diagonal entries of ${\rho}_{\ell}(X)$, the result holds.

\smallskip
If instead $\beta = 1$, then from our presentation, we have that
\begin{eqnarray*}
w^{an} = w^{t\alpha}, w^{2r}=1, w^{t} = w^{t(\alpha + \gamma)}.
\end{eqnarray*}

For $\phi = {\theta}_{\ell}$, we let ${\rho}_{\ell}$ be defined by
$$X \xrightarrow{{\rho}_{\ell}} \begin{pmatrix} w^{a\ell} & 0\\  0 & w^{a\ell + i \ell} \end{pmatrix}, Y \xrightarrow{{\rho}_{\ell}} \begin{pmatrix} 0 & w^{r \ell} \\  w^{r \ell} & 0 \end{pmatrix}, Z \xrightarrow{{\rho}_{\ell}}  \begin{pmatrix} w^{t\ell} & 0\\  0 & w^{t\ell} \end{pmatrix}$$
\smallskip

Again, ${\rho}_{\ell}$ has $R(\Gamma,V_{{\rho}_{\ell}}) \ncong {\mathbb{Z}}_p$ and any other such representation differs from ${\rho}_{\ell}$ by a character.  
\end{proof}

\medskip

\begin{remark}
Note that of course $D_{4n}$ is both dihedral and a central extension of $D_{2n}$ by ${\mathbb{Z}}/2{\mathbb{Z}}$.  We reconcile the results of \cite[Thm. 1.1 b]{meyer} with Lemma \ref{every} by pointing out that the faithful representation of $D_{4n}$,  ${\theta}_1$ is not trivial $Z = r^n$.  When we view $G = D_{4n}$ as a central extension of  $D_{2n}$ by ${\mathbb{Z}}/2{\mathbb{Z}} = \langle r^n \rangle$, the representation of $G$ inflated from the representation of $D_{2n}$ which we called $\theta_1$  is labeled $\theta_2$ when viewed as a representation of $D_{4n}$.
\end{remark}
\begin{corollary}
Given $p, n, m, G = \langle X, Y, Z \rangle , (\alpha, \beta, \gamma)$ as before.  By using universal deformation rings we show that in fact the center of $G$ must be exactly $\langle Z \rangle$.  
\end{corollary}
\begin{proof}
Suppose $G$ is given and let $\Gamma$ correspond to a faithful representation $\phi$ of $D_{2n} = G / \langle Z \rangle$.  That is, $\phi$ corresponds to a composition $$G \to D_{2n} = G / \langle Z \rangle \xrightarrow{{\theta}_k} N(D)$$  where ${\theta}_k$ is faithful.  From the proof of Lemma \ref{every}, we have seen that there exists a two-dimensional representation ${\rho}_{\ell}$ with $R(\Gamma,V_{{\rho}_{\ell}}) \ncong {\mathbb{Z}}_p $.  Then since $\phi$ is indecomposable by \cite{meyer}, it must therefore be trivial on the center of $G$ the result follows.  
\end{proof}
We note that this fact can be proved using alternative methods.\\

With Lemma \ref{every} in mind, we now compute the character group explicitly for $G$.

\begin{lemma}
\label{characters}
 Let $G$ be given by $\langle X, Y, Z \rangle , (\alpha, \beta, \gamma)$, and let $m' = \frac{m}{2}$.  Then, the characters of $G$ are as follow:
\begin{itemize}
\item[i.]
If $\beta = 0$, $2t' = -\alpha - \gamma$, then the characters of $G$ are $$X \xrightarrow{\chi_{a,b,c}} {(-1)}^a I^{-t'c}, Y \xrightarrow{\chi_{a,b,c}} {(-1)}^b, Z \xrightarrow{\chi_{a,b,c}} I^c $$
where $a, b = 0, 1$ and $c = 0, 1, ... m' - 1$ and $\displaystyle \langle I \rangle = {\mathbb{Z}}/m'{\mathbb{Z}}$.
\item[ii.]
If $\beta = 0, 2t' + 1 = -\alpha -\gamma$, then the characters of $G$ are $$X \xrightarrow{\chi_{a,b}} I^{a(-\alpha -\gamma)}, Y \xrightarrow{\chi_{a,b}} (-1)^b, Z \xrightarrow{\chi_{a,b}} I^{-2a} $$
where $a = 0, 1, ... m, b = 0, 1$ and $\displaystyle \langle I \rangle = {\mathbb{Z}}/m{\mathbb{Z}}$.
\item[iii.]
If $\beta = 1, 2t' = -\alpha - \gamma$, then the characters of $G$ are $$X \xrightarrow{\chi_{a,b}} (-1)^b I^{a(\alpha + \gamma)}, Y  \xrightarrow{\chi_{a,b}} I^a, Z\xrightarrow{\chi_{a,b}} I^{2a} $$
where $a = 0, 1, ... m - 1, b = 0, 1$ and $\langle I \rangle = {\mathbb{Z}}/m{\mathbb{Z}}$.
\item[iv.]
If $\beta = 1, 2t'+1 = -\alpha -\gamma$, then the characters of $G$ are $$X \xrightarrow{\chi_{a,b}} I^{a(-\alpha - \gamma)}, Y \xrightarrow{\chi_{a,b}} (-1)^b I^{-a}, Z \xrightarrow{\chi_{a,b}} I^{-2a} $$
where $a = 0, 1, ... m-1, b = 0, 1$ and $\langle I \rangle = {\mathbb{Z}}/m{\mathbb{Z}}$.
\end{itemize}
\end{lemma}
\begin{proof}
First, we compute the commutator subgroup $[G,G]$.  Suppose $\beta = 0$.  By direct computation, $[X,Y] = XYX^{n-1}Z^{-\gamma}Y = XYX^{n-1}YZ^{-\gamma} = XY^2 XZ^{-\alpha}Z^{-\gamma} = X^2 Z^{-\alpha -\gamma}$, so $\langle X^2 Z^{(-\alpha - \gamma)}\rangle \subseteq [G,G]$.  Since $[G,G] \lhd G$, $YX^2 Z^{(-\alpha - \gamma)}Y \in[G,G]$, but
\begin{eqnarray*}
YX^2 Z^{-\alpha - \gamma}Y &=& YXXYZ^{-\alpha - \gamma} \\
&=& X^{n-1}YXYZ^{\gamma -\alpha - \gamma}\\
&=& X^{n-1}YXYZ^{-\alpha }\\
&=& X^{n-1} X^{n-1} Z^{\gamma} Z^{-\alpha }\\
&=& Z^{2 \alpha}X^{n-2} Z^{-\alpha } Z^{\gamma-\alpha }\\
&=& X^{n-2} Z^{\gamma}\\
&=& {(X^2 Z^{(-\alpha - \gamma)})}^{-1}.\\
\end{eqnarray*}
Therefore $Y X^2 Z^{(-\alpha - \gamma)} = {(X^2 Z^{(-\alpha - \gamma)})}^{-1} Y$, so $\langle X^2 Z^{(-\alpha - \gamma)}\rangle \lhd G$.  Since $G / \langle X^2 Z^{(-\alpha - \gamma)}\rangle$ is abelian, $[G,G]=\langle X^2 Z^{(-\alpha - \gamma)}\rangle$.  A similar computation holds when $\beta = 1$.

Now, since $G$ admits a faithful two-dimensional complex representation, by Lemma \ref{identity}, we have that $\frac{n}{2}(\alpha + \gamma) = \alpha + \frac{m}{2} ( \textrm{ mod } m)$.  By the above computation, a counting argument shows that  $| G/[G,G]|$ is $4$ times the greatest common divisor of $ \frac{n}{2} (-\alpha - \gamma) + \alpha  (\textrm{ mod } m )$ and $m$.  Therefore the character group of $G$ has order $4 \frac{m}{2} = 2m$.  We will compute the characters explicitly.  

First, suppose  $\beta = 0$.  We solve the system $n X = \alpha Z, 2 Y = 0, 2 X = (\alpha + \gamma) Z, (m/2) Z = 0$ in ${\mathbb{Z}}-$mod.  Recall that $\frac{n}{2}(\alpha + \gamma) = \alpha + \frac{m}{2} ( \textrm{ mod } m)$.  Using row operations,

\vspace{.2 in}
 $$\begin{pmatrix} 2 & 0 & -\alpha - \gamma \\ n & 0 & - \alpha \\ 0 & 2 & 0\\ 0 & 0 & m/2 \end{pmatrix} \to \begin{pmatrix} 2 & 0 & -\alpha - \gamma \\ 0 & 0 & \frac{n}{2} (-\alpha - \gamma) + \alpha \\ 0 & 2 & 0\\ 0 & 0 & m/2 \end{pmatrix} \to \begin{pmatrix} 2 & 0 & -\alpha - \gamma \\ 0 & 2 & 0 \\  0 & 0 & m/2 \\ 0 & 0 & 0 \end{pmatrix}_.$$

\medskip 

Now, suppose $2 t' = -\alpha - \gamma$.  Then using column operations we obtain

\vspace{.2 in}
$$\begin{pmatrix} 2 & 0 & -\alpha - \gamma \\ 0 & 2 & 0 \\  0 & 0 & m/2 \\ 0 & 0 & 0 \end{pmatrix} \to \begin{pmatrix} 2 & 0 & 0 \\ 0 & 2 & 0 \\  0 & 0 & m/2 \\ 0 & 0 & 0 \end{pmatrix}_.$$

\medskip

\noindent
Thus, the character group of $G$ is isomorphic to ${\mathbb{Z}}/  2{\mathbb{Z}} \times {\mathbb{Z}}/  2{\mathbb{Z}} \times  {\mathbb{Z}}/  m'{\mathbb{Z}}$ with explicit generators $X Z^{t'}, Y$ and $Z$.  Solving for the images of $X, Y, Z$ respectively we obtain the result for $\beta = 0, 2t' = -\alpha -\gamma$.  The other cases are similar.
\end{proof}

For $\phi$ trivial on the center of $G$, ${\phi}= {\theta}_{{\ell}}$, for some representation ${\theta}_{{\ell}}$ of $D_{2n}$.  We now show that two representations $\phi_i = {\theta}_{{\ell}_i}$ for $i = 1, 2$, for which the parameter ${\ell}_i$ has the same greatest common divisor with $n$, have the same set of kernels of those representations whose corresponding ${\mathbb{F}}_p {{\Gamma}}_{\phi_i}$-modules have universal deformation ring not the $p$-adic integers.

\begin{proposition}
\label{diophantine}
Let  ${\phi}_i = {\theta}_{{\ell}_i}$ for $ i =1,2$ be two two-dimensional representations of $G= \langle X, Y, Z \rangle$ which are trival on $\langle Z \rangle$.  Then, if $({\ell}_1,n) = ({\ell}_2,n)$, then
$$\{ker(\theta), \textrm{where } R({\Gamma}_{{\theta}_1},V_{\theta}) \ncong \mathbb{Z}_p \}  = \{ker(\theta), \textrm{where } R({\Gamma}_{{\theta}_2},V_{\theta}) \ncong \mathbb{Z}_p \}. $$
\end{proposition}

\begin{proof}
Fix $\phi = {\theta}_{\ell}$.  Let ${\rho}_{\ell}$ denote the representation 
$$X \xrightarrow{{\rho}_{\ell}} \begin{pmatrix} w^{a\ell} & 0\\  0 & w^{a\ell + i \ell} \end{pmatrix}, Y \xrightarrow{{\rho}_{\ell}} \begin{pmatrix} 0 & w^r \\  w^r & 0 \end{pmatrix}, Z \xrightarrow{{\rho}_{\ell}}  \begin{pmatrix} w^{t\ell} & 0\\  0 & w^{t\ell} \end{pmatrix}.$$

We previously showed in the proof of Lemma \ref{every} that $R({\Gamma}_{\phi},V_{{\rho}_{\ell}}) \ncong {\mathbb{Z}}_p$.  Clearly,
$${\rho}_{\ell}(X^c Z^d) = \begin{pmatrix} w^{(ca\ell + \ell dt)} & 0\\0 & w^{(ca \ell +ci \ell + \ell dt)}\end{pmatrix}.$$

By Lemma \ref{every}, since the collection $\{ \rho \textrm{ with } R({\Gamma}_{\phi},V_{\rho}) \ncong \mathbb{Z}_p \}$ corresponds to the orbit of the representation ${\rho}_{\ell}$ under the action of the character group, the kernel of any such representation $\rho$ satisfies
$$ker(\rho) \subseteq \langle X^{\frac{p-1}{(p-1, i\ell)}}, Z \rangle = \langle X^{\frac{n}{(n,\ell)}}, Z \rangle = \{ g \in G \textrm{ with } {\rho}_{\ell}(g) \textrm{  a scalar matrix } \}.$$ 

Clearly, $\frac{n}{(n,\ell)}$ depends only on $(n,\ell)$.

\smallskip
Additionally,
$$ ker({\rho}_{\ell}) ={ \{ X^{\frac{n}{(n,\ell)}\sigma}Z^{\tau} \textrm{ such that } w^{t(\alpha \frac{\ell}{(n, \ell})\sigma+ \ell \tau)}=1\}}, $$
where $ 0 \leq \sigma \leq (n,\ell) -1,\textrm{ and } 0 \leq \tau \leq m-1. $

\smallskip

Therefore, the kernel of ${\rho}_{\ell}$ can be determined from the solutions to the linear diophantine equation 
\begin{eqnarray}
\alpha \frac{\ell}{(n, \ell)}\sigma+ \ell \tau
\end{eqnarray}

\noindent
where the coefficients $\alpha \frac{\ell}{(n, \ell)}$ and $ \ell$ are taken mod ($m$), and  $\sigma, \tau$ are in the appropriate intervals.  These intervals are independent of $\ell$, depending only on the greatest comon divisor of $\ell$ and $n$.  

Now, suppose $\beta = 0, 2t'+1= -\alpha -\gamma$.  Then, from our previous calculation in Lemma \ref{characters}, we see that the character group restricted to $\langle X, Z \rangle$ is cyclic with generator $\chi$ given by
$$X \xrightarrow{\chi} w^{t(-\alpha - \gamma)}, Z \xrightarrow{\chi} w^{-2t}.$$
Thus, 
$$ker(\chi \cdot {\rho}_{\ell}) ={ \{ X^{\frac{n}{(n,\ell)}\sigma}Z^{\tau} \textrm{ such that } w^{t(\alpha \frac{\ell}{(n, \ell)}\sigma+ \ell \tau+ \frac{n}{(n,\ell)}(-\alpha -\gamma)\sigma - 2 \tau)}=1\}}, $$
where $ 0 \leq \sigma \leq (n,\ell) -1, 0 \leq \tau \leq m-1 .$

\smallskip
Therefore, the kernel of $\chi \cdot {\rho}_{\ell}$ can be determined by the solutions to the linear diophantine equation 
\begin{eqnarray*}
\alpha\frac{\ell}{(n, \ell)}\sigma+ \ell \tau + \frac{n}{(n,\ell)}(-\alpha -\gamma)\sigma - 2 \tau \\
=[\alpha \frac{\ell}{(n, \ell)}+  \frac{n}{(n,\ell)}(-\alpha -\gamma)]\sigma + [\ell -2]\tau
\end{eqnarray*}

where again the coefficients are mod ($m$), and $\sigma, \tau$ are as above.  It's clear that to go from  ${\chi}^k \cdot {\rho}_{\ell}$ to ${\chi}^{k+1} \cdot {\rho}_{\ell}$  one simply adds $\frac{n}{(n,\ell)}(-\alpha -\gamma)\sigma - 2 \tau$ to the previous linear diophantine equation.  The key observation is that the equation $\frac{n}{(n,\ell)}(-\alpha -\gamma)\sigma - 2 \tau$ is independent of $\ell$, depending only on $(n, \ell)$.  

Thus given ${\ell}_1, {\ell}_2$ having the same greatest common divisor with $n$, we will show their corresponding lists of linear diophantine equations are the same up to permutation.  This is sufficient, as in both instances the variables $\sigma$ and $\tau$ correspond to the power of the same element of $G$, $\sigma$ corresponding to a power of $X^{\frac{n}{(n,\ell)}}$ and $\tau$ to a power of $Z$.  Explicitly, we show that for any $\ell$ the diophantine equation corresponding to the kernel of ${\rho}_{\ell}$, 
$$\alpha \frac{\ell}{(n, \ell)}\sigma+ \ell \tau = \frac{\ell}{(n, \ell)}(\alpha \sigma + (n, \ell ) \tau)$$
is the diophantine equation corresponding ${\chi}^k \cdot {\rho_{(n, \ell)}}$, for some $k$.  By the periodicity of the action of the character group this suffices.\\

First, suppose $2|\frac{n}{(n,\ell)}$.  Now we must only show $k \alpha \sigma + k (n, \ell) \tau$  is the diophantine equation corresponding to some power of $\chi$ acting on ${\rho}_{(n,\ell)}$.  Since $2|\frac{n}{(n,\ell)}$, we need only demonstrate this for $k$ odd.  But,
\begin{eqnarray*}
(n, \ell)[\frac{n}{(n,\ell)}(-\alpha -\gamma))\sigma -2\tau]&=&2[\frac{n}{2}(-\alpha -\gamma))\sigma]-2(n,\ell)\tau \\
&=&2[-\alpha + \frac{m}{2}]\sigma -2(n,\ell)\tau \\&=& -2\alpha -2(n,\ell) \tau.
\end{eqnarray*}

Therefore, any odd multiple of $\alpha \sigma + (n, \ell ) \tau$ can be achieved.

Now suppose $2 \not| \frac{n}{(n, \ell)}$.  We will show that there exists a $k$ with
$$k [\frac{n}{(n,\ell)}(-\alpha -\gamma))\sigma -2\tau]=\alpha \sigma + (n, \ell ) \tau.$$

Set $k = (\frac{m}{2} - \frac{(n,\ell)}{2})$.  Then, $(\frac{m}{2} - \frac{(n,\ell)}{2})[\frac{n}{(n,\ell)}(-\alpha -\gamma))\sigma -2\tau]$
\begin{eqnarray*}
 &=& (\frac{m}{2} - \frac{(n,\ell)}{2})[\frac{n}{(n,\ell)}(-\alpha -\gamma))\sigma + (\frac{m}{2} - \frac{(n,\ell)}{2})(-2\tau) \\ 
&=& [\frac{m}{2}\frac{n}{(n,\ell)}(-\alpha -\gamma) -\frac{n}{2}(-\alpha -\gamma)]\sigma + (n,\ell)\tau\\
&=& [\frac{m}{2}\frac{n}{(n,\ell)}(-\alpha -\gamma) + \alpha - \frac{m}{2}]\sigma + (n,\ell)\tau\\
&=& [\frac{m}{2}[\frac{n}{(n,\ell)}(-\alpha -\gamma) -1]  + \alpha]\sigma + (n,\ell)\tau\\
&=& \alpha \sigma + (n, \ell ) \tau
\end{eqnarray*}
as required since $ [\frac{n}{(n,\ell)}(-\alpha -\gamma) -1]$ is even.  Thus, we have showed that when $\beta = 0, 2t'+1= -\alpha -\gamma$, if ${\phi}_i = {\theta}_{{\ell}_i}$ for $i = 1,2 $ are such that $ (n, {\ell}_1) = (n,{\ell}_2) $,  then as sets
$$\{ker(\rho), \textrm{where } R({\Gamma}_{{\phi}_1},V_{\rho}) \ncong \mathbb{Z}_p \}  = \{ker(\rho), \textrm{where } R({\Gamma}_{{\phi}_2},V_{\rho}) \ncong \mathbb{Z}_p \}. $$

When $\beta = 1, 2t'+1=-\alpha -\gamma$, the identical argument holds.  In the remaining two cases, the argument is similar, though the group of characters restricted to $\langle  X, Z \rangle$ is not cyclic.  

\end{proof}

Before proving what amounts to the converse, we provide an example of the lists of diophantine equations for representations with the same greatest common divisor, as featured in the proof of Proposition \ref{diophantine}.

\begin{example}
Let (n,m) = (20,24), $(\alpha, \beta, \gamma) $= (18,0,9).  So $G$ is an extension of $D_{20}$ by ${\mathbb{Z}}/24{\mathbb{Z}}$, and $G$ is given by (18,0,9).  As the identity of Lemma \ref{identity} is satisfied, $G$ admits a faithful irreducible two-dimensional complex representation.  The collection of ${\theta}_{\ell}$ for which $(\ell,$n) = 1 is precisely ${\theta}_{\textrm{1}}, {\theta}_{\textrm{3}}, {\theta}_{\textrm{7}}, \textrm{ and } {\theta}_{\textrm{9}}$.  One can see that the linear diophantine equations corresponding to the kernels of the representations $\theta_{\ell}$ with $R({\Gamma}_{\phi}(V_{\theta_{\ell}}) \ncong {\mathbb{Z}}_p$ are the same up to permutation.  
\end{example}
\begin{tabular}{l*{4}{c}}
Representation & $\ell = 1$ & $\ell = 3$ & $\ell = 7$ & $\ell = 9$\\
\hline ${\rho}_{\ell}$ & $18 \sigma + \tau$ & $6\sigma + 3\tau$ & $6\sigma + 7\tau$ & $18\sigma + 9\tau$\\
  $\chi \cdot {\rho}_{\ell}$ & $6 \sigma + 23\tau$ & $18\sigma + \tau$ & $18\sigma + 5\tau$ & $6\sigma + 7\tau$\\
 ${\chi}^2 \cdot {\rho}_{\ell}$ & $18 \sigma + 21\tau$ & $6\sigma + 23\tau$ & $6\sigma + 3\tau$ & $18\sigma + 5\tau$\\
 ${\chi}^3 \cdot {\rho}_{\ell}$ & $6 \sigma + 19\tau$ & $18\sigma + 21\tau$ & $18\sigma + \tau$ & $6\sigma + 3\tau$\\
 ${\chi}^4 \cdot {\rho}_{\ell}$ & $18 \sigma + 17\tau$ & $6\sigma + 19\tau$ & $6\sigma + 23\tau$ & $18\sigma + \tau$\\
 ${\chi}^5 \cdot {\rho}_{\ell}$ & $6 \sigma + 15\tau$ & $18\sigma + 17\tau$ & $18\sigma + 21\tau$ & $6\sigma + 23\tau$\\
 ${\chi}^6 \cdot {\rho}_{\ell}$ & $18 \sigma + 13\tau$ & $6\sigma + 15\tau$ & $6\sigma + 19\tau$ & $18\sigma + 21\tau$\\
 ${\chi}^7 \cdot {\rho}_{\ell}$ & $6 \sigma + 11\tau$ & $18\sigma + 13\tau$ & $18\sigma + 17\tau$ & $6\sigma + 19\tau$\\
 ${\chi}^8 \cdot {\rho}_{\ell}$ & $18 \sigma + 9\tau$ & $6\sigma + 11\tau$ & $6\sigma + 15\tau$ & $18\sigma + 17\tau$\\
 ${\chi}^9 \cdot {\rho}_{\ell}$ & $6 \sigma + 7\tau$ & $18\sigma + 9\tau$ & $18\sigma + 13\tau$ & $6\sigma + 15\tau$\\
${\chi}^{10} \cdot {\rho}_{\ell}$ & $18 \sigma + 5\tau$ & $6\sigma + 7\tau$ & $6\sigma + 11\tau$ & $18\sigma + 13\tau$\\
${\chi}^{11} \cdot {\rho}_{\ell}$ & $6 \sigma + 3\tau$ & $18\sigma + 5\tau$ & $18\sigma + 9\tau$ & $6\sigma + 11\tau$\\
${\chi}^{12} \cdot {\rho}_{\ell}$ & $18 \sigma + \tau$ & $6\sigma + 3\tau$ & $6\sigma + 7\tau$ & $18\sigma + 9\tau$\\

\end{tabular}

\bigskip
Clearly, the lists above are the same up to cyclic permutation.  In each column, the next diophantine equation is obtained by adding $12\sigma -2\tau$ to the previous equation with coefficients viewed modulo $m$.  Note that different linear diophantine equations sometimes correspond to the same kernel.
\begin{remark}
The lists of diophantine equations may be identical for ${\ell}_1, {\ell}_2$ with different greatest common divisors with $n$.  Again let (n,m) = (20,24), $(\alpha, \beta, \gamma) $= (18,0,9).  Let ${\ell}_1 = 5$, ${\ell}_2 = 1$.  
\end{remark}
\noindent
Here ${\rho}_5 $ has diophantine equation $18 \sigma + 5\tau$ which appears in the table corresponding to $\ell$ with $(n,\ell) = 1$.  Moreover, for both $\ell = 1$ and $\ell = 5$, $$\frac{n}{(n,\ell)}(-\alpha -\gamma))\sigma -2\tau = 12\sigma -2 \tau \textrm{ modulo } m,$$ therefore their lists are identical, though the collection of kernels is not, since when $\ell = 5$ the value of $\sigma$ corresponds to a power of $X^4$.\\

We now prove that if $\phi$ is trivial on the center of $G$, then the fusion of $\phi$ can be determined by the collection $\{ker(\theta), \textrm{where } R({\Gamma}_{\phi},V_{\theta}) \ncong \mathbb{Z}_p \}$.

\begin{proposition}
\label{converse}
Suppose ${\phi}_1, {\phi}_2$ are trivial on $\langle Z \rangle = Z(G)$.  Then, if
$$\{ker(\rho), \textrm{where } R({\Gamma}_{{\phi}_1},V_{\rho}) \ncong \mathbb{Z}_p \}  = \{ker(\rho), \textrm{where } R({\Gamma}_{{\phi}_2},V_{\rho}) \ncong \mathbb{Z}_p \}, $$
then, ${\phi}_1,$ and ${\phi}_2$ have the same fusion.
\end{proposition}
\begin{proof}

Since $N = {\mathbb{Z}}/p{\mathbb{Z}}\times {\mathbb{Z}}/p{\mathbb{Z}}$ is abelian, it is clear that two distinct elements of $N$ are fused exactly when they are in the same orbit under $\phi$.  As $\phi$ is trivial on the center of $G$ as noted previously, $\phi$ is a representation for $D_{2n} =G/\langle Z \rangle$ inflated up to $G$.  It follows that the fusion of $\phi = \theta_{\ell}$ is exactly as calculated in \cite[Prop. 4.8]{meyer}.  Specifically, setting $\omega = w^i$, the fusion orbits are as follows:
\begin{enumerate}
\item $ Orb \left(\begin{array}{c}0\\0\end{array}\right) $ = $\left\{ \left(\begin{array}{c}0\\0\end{array}\right) \right\}_.$
%\item For $\left(\begin{array}{c}x\\y\end{array}\right) \in {\mathbb{F}_{p}}^{*} \times {\mathbb{F}_{p}}^{*}$ ,$ y/x \notin \langle \omega^{\ell} \rangle$ we have \\$ Orb \left(\begin{array}{c}x\\y\end{array}\right) $ = $\left\{ \left(\begin{array}{c}x\\y\end{array}\right)_, \left(\begin{array}{c}{\omega}^{\ell}x\\{\omega}^{-\ell}y\end{array}\right)_{,...}\left(\begin{array}{c}{\omega}^{(k - 1)\ell}x\\{\omega}^{-(k - 1)\ell}y\end{array}\right)_,\left(\begin{array}{c}y\\x\end{array}\right)_{,...}\left(\begin{array}{c}{\omega}^{-(k - 1)\ell}y\\{\omega}^{(k - 1)\ell}x\end{array}\right)\right\}$
\item For $\left(\begin{array}{c}x\\y\end{array}\right) \in \mathbb{F}_{p}^{*} \times \mathbb{F}_{p}^{*}$, $ y/x \in \langle \omega^{\ell} \rangle$, we have \\
$ Orb \left(\begin{array}{c}x\\y\end{array}\right) $ = $\left\{ \left(\begin{array}{c}x\\y\end{array}\right)_, \left(\begin{array}{c}{\omega}^{\ell}x\\{\omega}^{-\ell}y\end{array}\right)_{,...,}\left(\begin{array}{c}{\omega}^{(k - 1)\ell}x\\{\omega}^{-(k - 1)\ell}y\end{array}\right)\right\}_.$
\item For $\left(\begin{array}{c}x\\y\end{array}\right)$ not in 1. or 2., we have  \\ $ Orb \left(\begin{array}{c}x\\y\end{array}\right) $ = $\left\{ \left(\begin{array}{c}x\\y\end{array}\right)_, \left(\begin{array}{c}{\omega}^{\ell}x\\{\omega}^{-\ell}y\end{array}\right)_{,...,}\left(\begin{array}{c}{\omega}^{(k - 1)\ell}x\\{\omega}^{-(k - 1)\ell}y\end{array}\right)_,\left(\begin{array}{c}y\\x\end{array}\right)_{,...,}\left(\begin{array}{c}{\omega}^{-(k - 1)\ell}y\\{\omega}^{(k - 1)\ell}x\end{array}\right)\right\}_.$

\end{enumerate}

These orbits completely determine the fusion of $\phi$, as elements of $N$ are fused exactly when they are distinct elements of the same orbit.  In particular, we see that fusion depends only on the greatest common divisor of $\ell$ and $n$.  That is, ${\theta}_{\ell_1}$ and ${\theta}_{\ell_2}$ have the same fusion if and only if $(n,{\ell}_1) = (n,{\ell}_2)$ \\

Let $\phi = {\theta}_{\ell}$.  If $2|(n, \ell)$, by the main results in \cite{meyer} one can use the kernels of the representation of $G / \langle Z \rangle $ to determine fusion.  If $\ell$ is odd, however, there will not be any representations of  $G / \langle Z \rangle = D_{2n}$ with non-trivial universal deformation rings.  It suffices, therefore, to only distinguish between ${\theta}_{\ell}$, where $2\not| (n, \ell)$ .  It follows from the argument in Proposition \ref{diophantine}, that if $(n, \ell) = 1$, then for any $\theta$ with $R({\Gamma}_{\phi},V_{\theta}) \ncong {\mathbb{Z}}_p$, $ker(\theta) \subset \langle Z \rangle$.  Thus, the result follows once we demonstrate that for $\ell$ odd, there exists a representation $\theta$ with $R({\Gamma},V_{\theta}) \ncong {\mathbb{Z}}_p$ and $X^{\frac{n}{n,\ell}}Z^{\tau} \in ker(\theta)$, for some $\tau$.  As this corresponds to a solution to the diophantine equation for $\sigma = 1$, it is the minimal power of $X$ in any element of the kernel of any representation with non-trivial universal deformation ring.  Therefore, this minimal power of $X$ will determine $(n, \ell)$, and hence the fusion of $\phi$.  

But we know that $\alpha \sigma + (n, \ell) \tau $ is one of the diophantine equations corresponding to $\phi = {\theta}_{\ell}$ by Proposition \ref{diophantine}.  Since $(n, \ell)$ is odd, by repeatedly adding $\frac{n}{(n,\ell)}(-\alpha -\gamma))\sigma -2\tau$, we obtain the equation $B \sigma + 1 \tau$ for some $B$ in ${\mathbb{Z}}/m{\mathbb{Z}} $.  But $-B \in \langle 1 \rangle $, so therefore $(\sigma, \tau) = (1,-B)$ is a zero of this diophantine equation.  Therefore $X^{\frac{n}{n,\ell}}Z^{-B}$ is in the kernel of the appropriate representation $\theta$.  The result follows.
\end{proof}

\section{Proof of main results}
\label{proof}
We will now prove our main results.  First, let us relax our hypotheses from the last sections.  Now $G$ is any finite irreducible subgroup of $Gl_2({\mathbb{C}})$.  Again, we associate to $G$ on odd prime $p=p(G)$ such that ${\mathbb{F}}_p G$ is semisimple and ${\mathbb{F}}_p$ is sufficiently large for $G$.\\

We now prove Theorem \ref{Th1}.  The function $V \to R({\Gamma}_{\phi},V)$ is non-constant if and only if the representation $\phi$ is trivial on the center of $G$.
\begin{proof}
Recall that since the irreducible representation $G \subseteq Gl_2({\mathbb{C}})$ is realizable over ${\mathbb{F}}_p$ and $G$ is a semisimple subgroup of $Gl_2({\mathbb{F}}_p)$, if $\pi$ is the natural projection from $Gl_2({\mathbb{F}}_p) \xrightarrow{\pi} Gl_2({\mathbb{F}}_p)$ then,

$$\pi(G) =
\begin{cases}
D_{2n}, \textrm{ for some } n \\
{\mathbb{Z}}/m{\mathbb{Z}}, \textrm{ for some } m\\
A_4,
A_5, \textrm{ or }
S_4.
\end{cases} $$

 By \cite[Cor. 3.4]{meyer} any irreducible $\phi$ for which such a module $V$ exists must be trivial on the center of $G$.  Therefore, we need only prove the converse.  We proceed in cases.  If $\pi(G)$ is cyclic, $G$ is contained in a Cartan subgroup of $Gl_2({\mathbb{F}}_p)$ and is thus abelian.  Then, since ${\mathbb{F}}_p$ is sufficiently large, $G$ admits no irreducible faithful two-dimensional complex representations.  Therefore, no such irreducible $\phi$ exists (See \cite{meyer} for analysis of the function $V \to R(\Gamma, V)$ for $G$ finite, abelian).  For $G$ with $\pi(G)$ not cyclic, if $\phi$ is trivial on $Z(G)$ then $\phi$ is a representation of $D_{2n}, A_4, A_5$ or $S_4$.  Since $A_4$ and $A_5$ have no two-dimensional irreducible complex representations, for all irreducible finite subroups of $Gl_2({\mathbb{C}})$ that are central extensions of $A_4, \textrm{or } A_5$ by cyclic groups, the function $V \to R(\Gamma,V)$ is the constant function ${\mathbb{Z}}_p$, for every irreducible two-dimensional $\phi$. 

For a central extension of $S_4$ by a cyclic group which admits a faithful irreducible two-dimensional complex representation, the only $\phi$ irreducible and trivial on the center of $G$ is the representation of $D_6$ inflated first to $S_4$, then to $G$.  By inspection, $R({\Gamma}_{\phi}, V_{\phi}) \ncong {\mathbb{Z}}_p$. \\
Thus, we have reduced to the case when $\pi(G)$ dihedral.  Recall that the second cohomology of $D_{2n}$ with coefficients in ${\mathbb{Z}}/m{\mathbb{Z}}$ depends on the parity of $n, m$,
$$\charg{2}{D_{2n}}{{\mathbb{Z}}/m{\mathbb{Z}}} =
\begin{cases}
0, \textrm{ if } 2 \not| m\\
{\mathbb{Z}}/2{\mathbb{Z}}, \textrm{ if } 2 \not| n, 2|m\\
{\mathbb{Z}}/2{\mathbb{Z}}\times {\mathbb{Z}}/2{\mathbb{Z}} \times {\mathbb{Z}}/2{\mathbb{Z}}, \textrm{ if } 2|n, m.

\end{cases}$$
For $n$ odd, for any such $G$, one can find a representation of $D_{2n}$ with universal deformation ring not the $p$-adic integers by the main results in \cite{meyer}.  This representation can be inflated to $G$.  For $n$ even and $m$ odd since $ker(\pi)$ are scalar matrices, no such $G$ exists.  Therefore, by Lemma \ref{every} the result follows.
\end{proof}

We now prove Theorem \ref{Th2}.

\begin{proof}
Suppose $G \subseteq Gl_2({\mathbb{C}})$, $G$ finite and irreducible.  Suppose $\phi$ is irreducible and the function $V \to R(\Gamma,V_{\phi})$ is nonconstant.  Then, by the comments above, $G$ is a central extension of a dihedral group by a cyclic group,  or a central extension of $S_4$ by a cyclic group.  In the latter case, since $\phi$ is trivial on the center of $G$, $\phi$ must be the representation of $G$ inflated from the irreducible representation of $D_6 = S_4/({\mathbb{Z}}/2{\mathbb{Z}} \times {\mathbb{Z}}/2{\mathbb{Z}})$, so the result is trivial.  Thus, we reduce to the former case.   Let ${\mathbb{F}}_p$ be as above, and suppose $\pi(G)$ is dihedral.  We show that one can determine the fusion of ${\mathbb{Z}}/p{\mathbb{Z}}\times {\mathbb{Z}}/p{\mathbb{Z}}$ in $\Gamma$ from the graph of $V \to R(\Gamma,V)$ whenever this function is not constant on the collection of absolutely irreducible two-dimensional ${\mathbb{F}}_p \Gamma$-modules.\\

Again, by the comments inside the proof of Theorem \ref{Th1}, we reduce to the situation when both $n, m$ are even.  Therefore assume $2 | n, m,$ and $G$ is given by $\langle X, Y, Z \rangle$, $(\alpha, \beta, \gamma)$, where $\beta = 0, 1$. 

Since the function $V \to R({\Gamma}_{\phi},V)$ is not the constant function $V \to {\mathbb{Z}}_p$, $\phi$ is trivial on the center of $G$, so $\phi = {\theta}_{\ell}$ for some representation of $G/ \langle Z \rangle = D_{2n}$.  By Propositions \ref{diophantine}, \ref{converse},  and the proof of Theorem \ref{Th1}, the fusion of ${\mathbb{Z}}/p{\mathbb{Z}}\times {\mathbb{Z}}/p{\mathbb{Z}}$ in $\Gamma$ is determined by $(n,\ell) $.  That is, $\phi_1$ and $\phi_2$ have the same fusion exactly when $$\{ker(\rho), \textrm{where } R({\Gamma}_{{\phi}_1},V_{\rho}) \ncong \mathbb{Z}_p \}  = \{ker(\rho), \textrm{where } R({\Gamma}_{{\phi}_2},V_{\rho}) \ncong \mathbb{Z}_p \}.$$

Moreover, in the proof of Proposition \ref{converse} we see explicitly how to use the collection of kernels to see the greatest common divisor of any $\phi$ and hence its fusion.

Thus, the fusion of $\phi$ is uniquely determined by the set 
\begin{align*}
\{ker(\rho) : \rho \in \Sigma \textrm{ with } R(\Gamma,V_{\rho}) \ncong \mathbb{Z}_p\}.
\end{align*}
\end{proof}
Thus, the results of \cite{meyer} are generalized.  

%%%%%%%%%%%%%%%%%%%%%%%%%%%%%%%%%%%%%%%%%%%%%%%%%%%%%%%%%
%% Bibliography
%%%%%%%%%%%%%%%%%%%%%%%%%%%%%%%%%%%%%%%%%%%%%%%%%%%%%%%%%
\bibliographystyle{amsplain}

\end{document}